\documentclass[12pt]{article}
\usepackage{amssymb}
\usepackage{amsmath}
\usepackage{graphicx}

\newtheorem{theorem}{Theorem}

\newtheorem{definition}[theorem]{Definition}
\newtheorem{lemma}[theorem]{Lemma}
\newtheorem{remark}[theorem]{Remark}

\newenvironment{proof}[1][Proof]{\textbf{#1.} }{\ \rule{0.5em}{0.5em}}
\input{tcilatex}

\begin{document}

\title{Twisted Dedekind Type Sums Associated with Barnes' Type Multiple
Frobenius-Euler $l$-Functions}
\author{Mehmet Cenkci, Yilmaz Simsek, Mumun Can \ and Veli Kurt \\
Department of Mathematics, Akdeniz University, \\
07058-Antalya, Turkey \\
cenkci@akdeniz.edu.tr, ysimsek@akdeniz.edu.tr,\\
mcan@akdeniz.edu.tr, vkurt@akdeniz.edu.tr}
\date{}
\maketitle

\textbf{Abstract :} The aim of this paper is to construct new Dedekind type
sums. We construct generating functions of Barnes' type multiple
Frobenius-Euler numbers and polynomials. By applying Mellin transformation
to these functions, we define Barnes' type multiple $l$-functions, which
interpolate Frobenius-Euler numbers at negative integers. By using
generalizations of the Frobenius-Euler functions, we define generalized
Dedekind type sums and prove corresponding reciprocity law. We also give
twisted versions of the Frobenius-Euler polynomials and new Dedekind type
sums and corresponding reciprocity law. Furthermore, by using $p$-adic $q$%
-Volkenborn integral and twisted $(h,q)$-Bernoulli functions, we construct $p
$-adic $(h,q)$-higher order Dedekind type sums. By using relation between
Bernoulli and Frobenius-Euler functions, we also define analogues of
Hardy-Berndt type sums. We give some new relations related to to these sums
as well.

\bigskip


\bigskip


\section{Introduction, definitions and notations}

\setcounter{equation}{0} \setcounter{theorem}{0}

It is well-known that the classical Dedekind sums $s(h,k)$ first arose in
the transformation formula of the logarithm of the Dedekind-eta function. If 
$h$ and $k$ are coprime integers with $k>0$, Dedekind sums are defined by%
\begin{equation*}
s\left( h,k\right) =\sum_{a=1}^{k-1}\left( \left( \frac{a}{k}\right) \right)
\left( \left( \frac{ha}{k}\right) \right) ,
\end{equation*}%
where%
\begin{equation*}
\left( \left( x\right) \right) =\left\{ 
\begin{array}{ll}
x-[x]_{G}-\frac{1}{2}\text{,} & \text{if }x\text{ is not an integer} \\ 
0\text{,} & \text{if }x\text{ is an integer,}%
\end{array}%
\right.
\end{equation*}%
$[x]_{G}$ being the largest integer $\leq x$. The most important property of
Dedekind sums is the reciprocity law, which is given by%
\begin{equation*}
s(h,k)+s(k,h)=\frac{1}{12}\left( \frac{h}{k}+\frac{k}{h}+\frac{1}{hk}\right)
-\frac{1}{4}.
\end{equation*}%
For detailed information of Dedekind sums see (\cite{1}, \cite{5}, \cite{33}%
, \cite{34}, \cite{40}, \cite{3}, \cite{35}, \cite{2}, \cite{7}, \cite{19}, 
\cite{12}, \cite{11}, \cite{9}, \cite{8}, \cite{10}, \cite{simjnt2003dede}, 
\cite{simADCMgenDedeEisn}, \cite{simjmaaqDed}, \cite{simASCMsyang}, \cite{38}%
, \cite{37}).

In this paper, we define new Dedekind type sums related to Frobenius-Euler
functions as follows:

\begin{definition}
\label{def10}Let $n$, $h$ and $k$ be positive integers with $(h,k)=1$. We
define%
\begin{equation*}
S_{n,u}\left( h,k\right) =\sum_{a=0}^{k-1}u^{-\frac{ha}{k}}\frac{a}{k}%
\overline{H}_{n}\left( \frac{ha}{k},u\right) ,
\end{equation*}%
where $\overline{H}_{n}\left( \frac{ha}{k},u\right) $ denotes
Frobenius-Euler function, which is given by Definition \ref{def7}, and $u$
is an algebraic number $\neq 1$.
\end{definition}

The most important properties of these sums is the reciprocity law, which is
given by the following theorem.

\begin{theorem}
\label{th11}Let $n$, $h$ and $k$ be positive integers with $(h,k)=1$. Then,
we have%
\begin{eqnarray*}
&&\left( \frac{u^{k}}{1-u^{k}}k^{n}S_{n,u^{k}}\left( h,k\right) +\frac{u^{h}%
}{1-u^{h}}h^{n}S_{n,u^{h}}\left( k,h\right) \right) \\
&=&\sum_{j=0}^{n}\binom{n}{j}\frac{u^{k}}{1-u^{k}}H_{j}\left( u^{k}\right)
k^{j}\frac{u^{h}}{1-u^{h}}H_{n-j}\left( u^{h}\right) h^{n-j} \\
&&+\frac{1}{hk}\frac{u}{1-u}H_{n+1}\left( u\right) +\frac{u}{1-u}H_{n}\left(
u\right) ,
\end{eqnarray*}%
where $H_{n}\left( u\right) $ denotes Frobenius-Euler numbers given by (\ref%
{2,1}).
\end{theorem}

Throughout this paper, $\chi $ will denote a Dirichlet character of
conductor $f=f_{\chi }$, and $\chi _{0}$ will be a principle character with
conductor $f_{\chi _{0}}=1$. We also define Dedekind type sums attached to $%
\chi $ as follows:

\begin{definition}
\label{def19}Let $n$, $h$, $k$ be positive integers with $(h,k)=1$. Dedekind
type sums $S_{n,u^{k}}\left( h,k|\chi \right) $ are defined by%
\begin{equation*}
S_{n,u^{k}}\left( h,k|\chi \right)
=h^{n}\sum_{a=0}^{k-1}\sum_{b=0}^{h-1}\chi \left( kb+ha\right) u^{-\left(
kb+ha\right) }\frac{a}{k}\overline{H}_{n}\left( \frac{a}{k}+\frac{b}{h}%
,u^{hk}\right) .
\end{equation*}
\end{definition}

Note that if $\chi =\chi _{0}$ (that is, $f=1$), then%
\begin{equation*}
S_{n,u^{k}}\left( h,k|1\right) =\frac{u^{hk}-1}{u^{hk}}\frac{u^{k}}{u^{k}-1}%
S_{n,u^{k}}\left( h,k\right) .
\end{equation*}

We also note that the Definition \ref{def19} is different from Nagasaka
et.al's definition \cite{9}. In \cite{9}, Dedekind sums with character are
defined by using Bernoulli polynomials and Bernoulli function. In our
definition, we use Frobenius-Euler function $\overline{H}\left( x,u\right) $.

Reciprocity law of $S_{n,u^{k}}\left( h,k|\chi \right) $ is given by the
following theorem:

\begin{theorem}
\label{th4} Let $\chi $ be a Dirichlet character of conductor $f=f_{\chi }$
with $f|hk$. Let $n$, $h$ and $k$ be positive integers with $(h,k)=1$. Then,
we have%
\begin{eqnarray*}
&&\left( k^{n}S_{n,u^{k}}\left( h,k|\chi \right) +h^{n}S_{n,u^{h}}\left(
k,h|\chi \right) \right) \\
&=&\frac{1-u^{hk}}{u^{hk}}\frac{u^{f}}{1-u^{f}}\left( \frac{1}{hk}%
H_{n+1,\chi }\left( u\right) +H_{n,\chi }\left( u\right) \right) \\
&&+\frac{u^{hk}}{u^{hk}-1}\sum_{a=0}^{k-1}\sum_{b=0}^{h-1}\chi \left(
kb+ha\right) u^{-\left( kb+ha\right) } \\
&&\times \left( ^{1}H\left( u^{hk}\right) hk+\text{ }^{2}H\left(
u^{hk}\right) hk+kb+ha\right) ^{n}.
\end{eqnarray*}
\end{theorem}

Proofs of Theorem \ref{th11} and Theorem \ref{th4} are given in Section 2
and Section 3, respectively.

In this paper, $\mathbb{Z}_{p}$, $\mathbb{Q}_{p}$, $\mathbb{C}_{p}$ and $%
\mathbb{C}$ will, respectively, denote the ring of $p$-adic integers, the
field of $p$-adic rational numbers, the $p$-adic completion of the algebraic
closure of $\mathbb{Q}_{p}$ normalized by $\left| p\right| _{p}=p^{-1}$, and
the complex field. Let $q$ be an indeterminate such that if $q\in \mathbb{C}$
then $\left| q\right| <1$ and if $q\in \mathbb{C}_{p}$ then $\left|
1-q\right| _{p}<p^{-1/\left( p-1\right) }$, so that $q^{x}=$exp$\left( x%
\text{log}_{p}q\right) $ for $\left| x\right| _{p}\leq 1$, where $\log _{p}$
is the Iwasawa $p$-adic logarithm function (\cite[Chap.4]{13}, \cite{14}, 
\cite{25}, \cite{26}, \cite{sirivastKimSimsek}). We use the notation%
\begin{equation*}
\left[ x\right] =\left[ x:q\right] =\frac{1-q^{x}}{1-q},
\end{equation*}%
so that $\lim_{q\rightarrow 1}\left[ x\right] =x$.

The $p$-adic $q$-integral (or $q$-Volkenborn integral) is originally
constructed by Kim \cite{25}. Kim indicated a connection between the $q$%
-Volkenborn integral and non-Archimedean combinatorial analysis. The $p$%
-adic $q$-Volkenborn integral is used in mathematical physics, derivation of
the functional equation of the $q$-zeta function and the $q$-Stirling
numbers, and the $q$-Mahler theory of integration with respect to a ring $%
\mathbb{Z}_{p}$ together with Iwasawa's $p$-adic $q$-$L$-function. Recently,
many applications of the $q$-Volkenborn integral have studied by \ the
authors \cite{cenkcicankurt1}, \cite{simjkms2006}, \cite{30}, \cite%
{sirivastKimSimsek}, and many mathematicians.

We give some basic properties of $p$-adic $q$-Volkenborn integral as follows:

For $g\in UD(\mathbb{Z}_{p},\mathbb{C}_{p})=\left\{ g\mid g:\mathbb{Z}%
_{p}\rightarrow \mathbb{C}_{p}\text{ is uniformly differentiable function}%
\right\} $, the $p$-adic $q$-Volkenborn integral is defined by \cite{14}, 
\cite{25}, \cite{26}%
\begin{equation*}
I_{q}\left( g\right) =\int\limits_{\mathbb{Z}_{p}}g\left( x\right) d\mu
_{q}\left( x\right) =\lim_{N\rightarrow \infty }\frac{1}{\left[ p^{N}:q%
\right] }\sum_{x=0}^{p^{N}-1}g\left( x\right) q^{x},
\end{equation*}%
where%
\begin{equation*}
\mu _{q}\left( x+p^{N}\mathbb{Z}_{p}\right) =\frac{q^{x}}{\left[ p^{N}:q%
\right] }
\end{equation*}%
is the $q$-analogue of the Haar measure. For the limiting case $q=1$,%
\begin{equation*}
I_{1}\left( g\right) =\lim_{q\rightarrow 1}I_{q}\left( g\right)
=\int\limits_{\mathbb{Z}_{p}}g\left( x\right) d\mu _{1}\left( x\right)
=\lim_{N\rightarrow \infty }\frac{1}{p^{N}}\sum_{x=0}^{p^{N}-1}g\left(
x\right) ,
\end{equation*}%
with%
\begin{equation*}
\mu _{1}\left( x+p^{N}\mathbb{Z}_{p}\right) =\frac{1}{p^{N}}
\end{equation*}%
is the Haar measure. If $g_{1}\left( x\right) =g\left( x+1\right) $, then%
\begin{equation}
I_{1}\left( g_{1}\right) =I_{1}\left( g\right) +g^{\prime }\left( 0\right) ,
\label{1,1}
\end{equation}%
where $g^{\prime }\left( 0\right) =\left. \frac{d}{dx}g\left( x\right)
\right| _{x=0}$ (\cite{25}, \cite{26}).

Let $f$ be any fixed positive integer with $\left( p,f\right) =1$. Then set%
\begin{eqnarray*}
\mathbb{X} &\mathbb{=}&\mathbb{X}_{f}=\underset{N}{\underleftarrow{\text{lim}%
}}\left( \mathbb{Z}/fp^{N}\mathbb{Z}\right) ,\text{ }\mathbb{X}_{1}=\mathbb{Z%
}_{p}, \\
\mathbb{X}^{\ast } &=&\bigcup\limits_{0<a<fp}a+fp^{n}\mathbb{Z}_{p}, \\
a+fp^{n}\mathbb{Z}_{p} &=&\left\{ x\in \mathbb{X}:x\equiv a\left( \text{mod}%
fp^{N}\right) \right\} ,
\end{eqnarray*}%
where $a\in \mathbb{Z}$ with $0\leq a<fp^{N}$. Note that%
\begin{equation*}
\int\limits_{\mathbb{Z}_{p}}g\left( x\right) d\mu _{1}\left( x\right)
=\int\limits_{\mathbb{X}}g\left( x\right) d\mu _{1}\left( x\right)
\end{equation*}%
for $g\in UD\left( \mathbb{Z}_{p},\mathbb{C}_{p}\right) $ (\cite{25}, \cite%
{26}).

Let%
\begin{equation*}
T_{p}=\bigcup\limits_{n>1}C_{p^{n}}=\underset{n\rightarrow \infty }{\text{lim%
}}C_{p^{n}},
\end{equation*}%
where $C_{p^{n}}=\left\{ \zeta :\zeta ^{p^{n}}=1\right\} $ is the cyclic
group of order $p^{n}$. For $\zeta \in T_{p}$, the function $x\mapsto \zeta
^{x}$ is a locally constant function from $\mathbb{Z}_{p}$ to $\mathbb{C}%
_{p} $ (\cite{29}, \cite{28}). By using $q$-Volkenborn integration, the
second author \cite{30} defined generating function of twisted $\left(
h,q\right) $-extension of Bernoulli numbers $B_{n,\zeta }^{\left( h\right)
}\left( q\right) $ and polynomials $B_{n,\zeta }^{\left( h\right) }\left(
x,q\right) $ by means of%
\begin{equation}
\frac{h\log q+t}{\zeta q^{h}e^{t}-1} =\sum_{n=0}^{\infty }B_{n,\zeta
}^{\left( h\right) }\left( q\right) \frac{t^{n}}{n!}, \text{ and } \frac{%
h\log q+t}{\zeta q^{h}e^{t}-1}e^{xt} =\sum_{n=0}^{\infty }B_{n,\zeta
}^{\left( h\right) }\left( x,q\right) \frac{t^{n}}{n!},  \label{y1}
\end{equation}%
\noindent respectively. Note that the numbers $B_{n,\zeta }^{\left( h\right)
}\left( q\right) $ are given by \cite{30}%
\begin{equation*}
B_{0,\zeta }^{\left( h\right) }\left( q\right) =\frac{h\log q}{\zeta q^{h}-1}%
\text{ and }\zeta q^{h}\left( B_{\zeta }^{\left( h\right) }\left( q\right)
+1\right) ^{n}-B_{n,\zeta }^{\left( h\right) }\left( q\right) =\delta _{n,1},
\end{equation*}%
with the usual convention about replacing $\left( B_{\zeta }^{\left(
h\right) }\left( q\right) \right) ^{j}$ by $B_{j,\zeta }^{\left( h\right)
}\left( q\right) $ in the binomial expansion, where $\delta _{n,1}$ is the
Kronecker symbol. If $\zeta \rightarrow 1$, $B_{j,\zeta }^{\left( h\right)
}\left( q\right) \rightarrow B_{j}^{\left( h\right) }\left( q\right) $,
which are the numbers defined by Kim \cite{277}.

In $p$-adic case, by using $p$-adic $q$-Volkenborn integral and twisted $%
\left( h,q\right) $-Bernoulli functions, we construct $p$-adic $\left(
h,q\right) $-higher order Dedekind type sums as follows:

\begin{definition}
\label{def12}Let $h$, $a$ and $b$ be fixed integers with $\left( a,b\right)
=1$, and let $p$ be an odd prime such that $p|b$. For $\zeta \in T_{p}$, we
define twisted $\left( h,q\right) $-Dedekind type sums as%
\begin{equation*}
s_{m,\zeta }^{\left( h\right) }\left( a,b:q\right) =\sum_{j=0}^{b-1}\frac{j}{%
b}\int\limits_{\mathbb{Z}_{p}}q^{hx}\zeta ^{x}\left( x+\left\{ \frac{ja}{b}%
\right\} \right) ^{m}d\mu _{1}\left( x\right) ,
\end{equation*}%
\noindent where $\left\{ t\right\} $ denotes the fractional part of a real
number $t$.
\end{definition}

\noindent Observe that when $h=1$, $q\rightarrow 1$ and $\zeta \rightarrow 1$%
, the sum $s_{m,1}^{\left( 1\right) }\left( a,b:1\right) $ reduces to $p$%
-adic analogue of higher order Dedekind sums $b^{m}s_{m+1}\left( a,b\right) $%
, defined by Apostol \cite{1}. The main properties of $s_{m,\zeta }^{\left(
h\right) }\left( a,b:q\right) $ will be given in Section 5.

Dedekind sums were generalized by various mathematicians. Here, we list some
of them. Apostol \cite{1} defined generalized Dedekind sums $s_{n}\left(
h,k\right) $ by%
\begin{equation}
s_{n}(h,k)=\sum\limits_{a=1}^{k-1}\frac{a}{k}\overline{B}_{n}\left( \frac{ha%
}{k}\right) ,  \label{A1}
\end{equation}%
where $n,h,k$ are positive integers and $\overline{B}_{n}\left( x\right) $
is the $n$th Bernoulli function, which is defined as follows:

\begin{equation}
\overline{B}_{n}\left( x\right) =B_{n}\left( x-\left[ x\right] _{G}\right) ,%
\text{ if }n>1,  \label{A2}
\end{equation}%
\begin{equation*}
\overline{B}_{1}\left( x\right) =\left\{ 
\begin{array}{ll}
B_{1}\left( x-\left[ x\right] _{G}\right) , & \text{if }x\notin \mathbb{Z}
\\ 
0, & \text{if }x\in \mathbb{Z},%
\end{array}%
\right.
\end{equation*}

\noindent where $B_{n}\left( x\right) $ is the Bernoulli polynomial \cite{1}%
, \cite{simjmaaqDed}, \cite{sirivastKimSimsek}. For odd values of $n$, these
generalized Dedekind sums satisfy a reciprocity law, first proved by Apostol 
\cite{1}:%
\begin{eqnarray*}
\left( n+1\right)&&\left\{ hk^{n}s_{n}(h,k)+kh^{n}s_{n}(k,h)\right\} \\
&=&\sum\limits_{j=0}^{n+1}\binom{n+1}{j}%
(-1)^{j}B_{j}h^{j}B_{n+1-j}k^{n+1-j}+nB_{n+1},
\end{eqnarray*}%
\noindent where $\left( h,k\right) =1$ and $B_{n}$ is the $n$th Bernoulli
number. Berndt \cite{5} gave a character transformation formula similar to
those for the Dedekind $\eta $-function and defined Dedekind sums with
character $s(h,k;\chi )$ by 
\begin{equation*}
s(h,k;\chi )=\sum\limits_{a=0}^{kf-1}\chi (a)\overline{B}_{1,\chi }\left( 
\frac{ha}{k}\right) \overline{B}_{1}\left( \frac{a}{kf}\right) ,
\end{equation*}%
for $(h,k)=1$. Here, $\chi $ denotes a primitive character of conductor $f$
and $\overline{B}_{n,\chi }\left( x\right) $ is the character Bernoulli
function defined as $\overline{B}_{n,\chi }\left( x\right) =B_{n,\chi
}\left( x\right) $ for $0<x<1$, where $B_{n,\chi }\left( x\right) $ are the
character Bernoulli polynomials which are defined as follows (\cite{5}, \cite%
{sirivastKimSimsek}):%
\begin{equation*}
\sum\limits_{a=0}^{f-1}\frac{\chi (a)te^{(a+x)t}}{e^{ft}-1}%
=\sum\limits_{n=0}^{\infty }B_{n,\chi }\left( x\right) \frac{t^{n}}{n!}.
\end{equation*}%
\noindent In \cite{7}, Gunnells and Sczech defined certain
higher-dimensional Dedekind sums that generalize the classical Dedekind
sums. By using Barnes' double zeta function Ota \cite{8} defined derivatives
of Dedekind sums and proved their reciprocity laws. Using similar method,
Nagasaka et.al \cite{9} gave further generalizations of generalized Dedekind
sums. Cenkci et.al \cite{35} gave degenerate analogues of classical Dedekind
sums and exact generalizations of Berndt's character Dedekind sums to the
case of any positive number. By using the $p$-adic interpolation of certain
partial zeta functions, Rosen and Snyder \cite{10} defined $p$-adic Dedekind
sums in the sense of Apostol \cite{1}. They also established the reciprocity
law for these new $p$-adic Dedekind sums via interpolation of corresponding
law for generalized Dedekind sums. In \cite{12} and \cite{11}, Kudo extended
the results of Rosen and Snyder. He defined $p$-adic continuous function
which interpolates higher-order Dedekind sums. In \cite{14}, \cite{16}, Kim
defined $q$-Bernoulli numbers $\beta _{n}\left( q\right) \in \mathbb{C}$ and 
$q$-Bernoulli polynomials $\beta _{n}\left( x,q\right) $ which are different
Carlitz's $q$-Bernoulli numbers \cite{17}, \cite{18}. By using these
polynomials and an invariant $p$-adic $q$-Volkenborn integral on $\mathbb{Z}%
_{p}$, he constructed a $p$-adic $q$-analogue of generalized Dedekind sums $%
b^{m}s_{m+1}\left( a,b\right) $.

In \cite{simjmaaqDed}, the second author defined new generating functions.
By using these functions, he constructed $q$-Dedekind type sums related to
Apostol's Dedekind type sums \cite{1}. By using $p$-adic $q$-Volkenborn
integral, he \cite{simjkms2006} constructed $p$-adic $q$-higher-order Hardy
type sums.

In \cite{277}, Kim constructed the new $\left( h,q\right) $-extension of the
Bernoulli numbers and polynomials. By applying Mellin transformation to the
generating function of the $\left( h,q\right) $- Bernoulli numbers, he
defined $(h,q)$-zeta functions and $(h,q)$-$L$-functions, which interpolate $%
\left( h,q\right) $- Bernoulli numbers at negative integers. By using $p$%
-adic $q$-Volkenborn integral, the distribution property of twisted $\left(
h,q\right) $-Bernoulli polynomials is given by the following theorem:

\begin{theorem}
\label{th3}(\cite{30}) For any positive integer $m$,%
\begin{equation}
B_{n,\zeta }^{\left( h\right) }\left( x,q\right)
=m^{n-1}\sum_{a=0}^{m-1}\zeta ^{a}q^{ha}B_{n,\zeta ^{m}}^{\left( h\right)
}\left( \frac{a+x}{m},q^{m}\right)  \label{Y1}
\end{equation}%
for all integers $n\geq 0$.
\end{theorem}

\noindent Observe that for $\zeta \rightarrow 1$, $q\rightarrow 1$ and $h=1$%
, we have%
\begin{equation}
m^{n-1}\sum_{j=0}^{m-1}B_{n}\left( x+\frac{j}{m}\right) =B_{n}\left(
mx\right) .  \label{Y2*}
\end{equation}

The second author \cite{30} gave generating function for twisted $\left(
h,q\right) $-extensions of generalized Bernoulli numbers and polynomials
associated with a Dirichlet character $\chi $ as follows:%
\begin{eqnarray*}
\sum_{a=1}^{f}\frac{\chi \left( a\right) \zeta ^{a}q^{ha}e^{at}\left( h\log
q+t\right) }{\zeta ^{f}q^{hf}e^{tf}-1} &=&\sum_{n=0}^{\infty }B_{n,\zeta
,\chi }^{\left( h\right) }\left( q\right) \frac{t^{n}}{n!}, \\
\sum_{a=1}^{f}\frac{\chi \left( a\right) \zeta ^{a}q^{ha}e^{\left(
a+x\right) t}\left( h\log q+t\right) }{\zeta ^{f}q^{hf}e^{tf}-1}
&=&\sum_{n=0}^{\infty }B_{n,\zeta ,\chi }^{\left( h\right) }\left(
x,q\right) \frac{t^{n}}{n!}.
\end{eqnarray*}%
\noindent Note that%
\begin{eqnarray}
B_{n,\zeta ,\chi }^{\left( h\right) }\left( q\right)
&=&f^{n-1}\sum_{j=1}^{f}\chi \left( j\right) \zeta ^{j}q^{hj}B_{n,\zeta
}^{\left( h\right) }\left( \frac{j}{f},q^{f}\right) ,  \notag \\
B_{n,\zeta ,\chi }^{\left( h\right) }\left( x,q\right)
&=&f^{n-1}\sum_{j=1}^{f}\chi \left( j\right) \zeta ^{j}q^{hj}B_{n,\zeta
}^{\left( h\right) }\left( \frac{j+x}{f},q^{f}\right) .  \label{1,4}
\end{eqnarray}%
\noindent Using $q$-Volkenborn integration, Witt's type formulas for these
numbers and polynomials were also given by \cite{30}%
\begin{eqnarray}
B_{n,\zeta }^{\left( h\right) }\left( q\right) &=&\int\limits_{\mathbb{Z}%
_{p}}\zeta ^{x}q^{hx}x^{n}d\mu _{1}\left( x\right) ,  \label{Y4} \\
B_{n,\zeta }^{\left( h\right) }\left( x,q\right) &=&\int\limits_{\mathbb{Z}%
_{p}}\zeta ^{t}q^{ht}\left( x+t\right) ^{n}d\mu _{1}\left( t\right) ,  \notag
\\
B_{n,\zeta ,\chi }^{\left( h\right) }\left( q\right) &=&\int\limits_{\mathbb{%
X}}\chi \left( x\right) \zeta ^{x}q^{hx}x^{n}d\mu _{1}\left( x\right) . 
\notag
\end{eqnarray}%
\noindent We note that, if $\zeta \rightarrow 1$ then, $B_{n,\zeta ,\chi
}^{\left( h\right) }\left( q\right) \rightarrow B_{n,\chi }^{\left( h\right)
}\left( q\right) $ and $B_{n,\zeta ,\chi }^{\left( h\right) }\left(
x,q\right) \rightarrow B_{n,\chi }^{\left( h\right) }\left( x,q\right) $
which are defined by Kim \cite{277}.

Now we summarize our paper as follows:

In Section 2, we construct new generating functions of Frobenius-Euler
numbers and polynomials. We give relations between these numbers and
polynomials. We also define generating functions of Barnes' type multiple
Frobenius-Euler numbers and polynomials. By applying Mellin transformation
to these functions, we construct Barnes' type multiple $l$-functions. We
define Dedekind type sums related to the Frobenius-Euler functions. We prove
reciprocity laws of these sums. In Section 3, by using Dirichlet character,
we give generalizations of the Frobenius-Euler polynomials and numbers. We
construct generalized Dedekind type sums and prove corresponding reciprocity
law. In Section 4, we give twisted versions of new Dedekind type sums and
corresponding reciprocity law. In Section 5, by using $p$-adic $q$%
-Volkenborn integral and twisted $\left( h,q\right) $-Bernoulli functions,
we construct $p$-adic $\left( h,q\right) $-higher order Dedekind type sums.
By using relation between Bernoulli and Frobenius-Euler functions, we also
define new Hardy-Berndt type sums. We give some new relations related to to
these sums as well.

\section{New Dedekind Type Sums in the Complex Case}

\setcounter{equation}{0} \setcounter{theorem}{0}

Let $F_{u}\left( t\right) $ be the generating function of Frobenius-Euler
numbers $H_{n}\left( u\right) $, that is,%
\begin{equation}
F_{u}\left( t\right) =\sum_{n=0}^{\infty }H_{n}\left( u\right) \frac{t^{n}}{%
n!}=\frac{1-u}{e^{t}-u},  \label{2,1}
\end{equation}%
(\cite{22}, \cite{23}, \cite{20}, \cite{21}, \cite{simBKMStwEulnu}, \cite%
{sirivastKimSimsek}, \cite{32}). The generating function of Frobenius-Euler
polynomials $H_{n}\left( x,u\right) $ can be defined in a natural way by%
\begin{equation}
F_{u}\left( x,t\right) =F_{u}\left( t\right) e^{xt}=\sum_{n=0}^{\infty
}H_{n}\left( x,u\right) \frac{t^{n}}{n!}=\frac{1-u}{e^{t}-u}e^{xt}.
\label{2,2}
\end{equation}%
\noindent Now rewriting $F_{u}\left( x,t\right) $, we have%
\begin{eqnarray*}
F_{u}\left( x,t\right) &=&\sum_{n=0}^{\infty }H_{n}\left( x,u\right) \frac{%
t^{n}}{n!}=\frac{1-u}{e^{t}-u}e^{xt}=u^{-1}\left( u-1\right)
e^{xt}\sum_{n=0}^{\infty }\left( u^{-1}e^{t}\right) ^{n} \\
&=&\sum_{n=0}^{\infty }u^{-n-1}\left( u-1\right) e^{\left( n+x\right) t}.
\end{eqnarray*}%
\noindent By applying Mellin transform to $F_{u}\left( x,t\right) $,

\begin{eqnarray*}
\frac{1}{\Gamma \left( s\right) }\int\limits_{0}^{\infty }t^{s-1}F_{u}\left(
x,-t\right) dt &=&\frac{1}{\Gamma \left( s\right) }\sum_{n=0}^{\infty
}u^{-n-1}\left( u-1\right) \int\limits_{0}^{\infty }t^{s-1}e^{-\left(
n+x\right) t}dt \\
&=&\frac{u-1}{u}\sum_{n=0}^{\infty }\frac{u^{-n}}{\left( n+x\right) ^{s}},
\end{eqnarray*}%
\noindent where $\Gamma \left( s\right) $ is the Euler gamma function. The $%
l $-function which interpolates Frobenius-Euler numbers at negative integer
values, is defined by%
\begin{equation}
l\left( s;u\right) =\sum_{n=1}^{\infty }\frac{u^{-n}}{n^{s}}  \label{2,3}
\end{equation}%
for Re$\left( s\right) >1$ and $u\in \mathbb{C}$ with $\left| u\right| \geq
1 $. Two-variable $l$-function is defined by%
\begin{equation}
l\left( s,x;u\right) =\sum_{n=0}^{\infty }\frac{u^{-n}}{\left( n+x\right)
^{s}}  \label{2,4}
\end{equation}%
for $x\neq $ zero or negative integer, Re$\left( s\right) >1$ and $u\in 
\mathbb{C}$ with $\left| u\right| \geq 1$. So defined two-variable $l$%
-function interpolates Frobenius-Euler polynomials $H_{n}\left( x,u\right) $%
. Indeed, we have%
\begin{equation*}
\frac{u-1}{u}\sum_{n=0}^{\infty }\frac{u^{-n}}{\left( n+x\right) ^{s}}=\frac{%
u-1}{u}l\left( s,x;u\right) =\frac{1}{\Gamma \left( s\right) }%
\int\limits_{0}^{\infty }t^{s-1}F_{u}\left( x,-t\right) dt.
\end{equation*}%
\noindent For $s=-n$, $n\in \mathbb{Z}$, $n\geq 0$, by using Cauchy residue
theorem, we have%
\begin{equation}
\frac{u-1}{u}l\left( -n,x;u\right) =H_{n}\left( x,u\right) .  \label{2,9}
\end{equation}

In \cite{24}, Kim gave the definition of $r$th Frobenius-Euler polynomials
of $x$ with parameters $a_{1},\ldots ,a_{r}$ as%
\begin{equation*}
\frac{\left( 1-u\right) ^{r}}{\prod\limits_{j=1}^{r}\left(
e^{a_{j}t}-u\right) }e^{xt}=\sum_{n=0}^{\infty }H_{n}^{\left( r\right)
}\left( x,u|a_{1},\ldots ,a_{r}\right) \frac{t^{n}}{n!},
\end{equation*}%
\noindent for complex numbers $x,a_{1},\ldots ,a_{r},u$ such that $a_{j}\neq
0$ for each $j=1,\ldots ,r$ and $\left| u\right| >1$. For $x=0$, the $r$th
Frobenius-Euler polynomials are called as the $r$th Frobenius-Euler numbers
and denoted by $H_{n}^{\left( r\right) }\left( 0,u|a_{1},\ldots
,a_{r}\right) =H_{n}^{\left( r\right) }\left( u|a_{1},\ldots ,a_{r}\right) $.

Let $x$ be a complex number, Re$\left( x\right) >0$ and $a_{1},\ldots ,a_{r}$
be real numbers such that $a_{j}\neq 0$ for each $j=0,\ldots ,r$. We modify
the definition of $r$th Frobenius-Euler polynomials of $x$ with parameters $%
a_{1},\ldots ,a_{r}$ as%
\begin{equation*}
\prod_{j=1}^{r}\frac{1-u^{a_{j}}}{e^{a_{j}t}-u^{a_{j}}}e^{xt}=\sum_{n=0}^{%
\infty }H_{r,n}\left( x,u|a_{1},\ldots ,a_{r}\right) \frac{t^{n}}{n!}.
\end{equation*}%
\noindent Note that for $r=1$, $H_{1,n}\left( x,u|a_{1}\right) =H_{n}\left(
x,u^{a_{1}}\right) $. We have the following identity about $H_{r,n}\left(
x,u|a_{1},\ldots ,a_{r}\right) $:%
\begin{equation*}
H_{r,n}\left( x,u|a_{1},\ldots ,a_{r}\right) =\left( ^{1}H\left(
u^{a_{1}}\right) a_{1}+\cdots +^{r}H\left( u^{a_{r}}\right) a_{r}+x\right)
^{n},
\end{equation*}%
\noindent where in the multinomial expansion of $\left( ^{1}H\left(
u^{a_{1}}\right) a_{1}+\cdots +^{r}H\left( u^{a_{r}}\right) a_{r}+x\right)
^{n}$ we mean that%
\begin{equation*}
\left( ^{i}H\left( u\right) \right) ^{j}=H_{j}\left( u\right) \text{ but }%
\left( ^{i}H\left( u\right) \right) ^{j}\left( ^{l}H\left( u\right) \right)
^{k}\neq H_{j+k}\left( u\right) \text{ if }i\neq l\text{.}
\end{equation*}%
\noindent This identity can be shown by using the definition of
Frobenius-Euler numbers (\ref{2,1}):%
\begin{eqnarray*}
\prod_{j=1}^{r}\frac{1-u^{a_{j}}}{e^{a_{j}t}-u^{a_{j}}}e^{xt}
&=&\prod_{j=1}^{r}\left( \sum_{n_{j}=0}^{\infty }H_{n_{j}}\left(
u^{a_{j}}\right) \frac{\left( a_{j}t\right) ^{n_{j}}}{n_{j}!}\right) \left(
\sum_{n=0}^{\infty }\frac{\left( xt\right) ^{n}}{n!}\right) \\
&=&\sum_{N=0}^{\infty }\frac{\left( ^{1}H\left( u^{a_{1}}\right)
a_{1}+\cdots +^{r}H\left( u^{a_{r}}\right) a_{r}+x\right) ^{N}t^{N}}{N!}.
\end{eqnarray*}

Let $F_{r,u}\left( x,t\right) $ be the generating function of $H_{r,n}\left(
x,u|a_{1},\ldots ,a_{r}\right) $. Then, we have%
\begin{eqnarray*}
F_{r,u}\left( x,t\right) &=&\sum_{n=0}^{\infty }H_{r,n}\left(
x,u|a_{1},\ldots ,a_{r}\right) \frac{t^{n}}{n!}=\prod_{j=1}^{r}\frac{%
1-u^{a_{j}}}{e^{a_{j}t}-u^{a_{j}}}e^{xt} \\
&=&\prod_{j=1}^{r}\left( 1-u^{-a_{j}}\right) e^{xt}\sum_{n_{j}=0}^{\infty
}u^{-a_{j}n_{j}}e^{-a_{j}n_{j}t} \\
&=&\prod_{j=1}^{r}\left( 1-u^{-a_{j}}\right) \sum_{n_{1},\ldots
,n_{r}=0}^{\infty }u^{-\left( n_{1}a_{1}+\cdots +n_{r}a_{r}\right)
}e^{-\left( x+n_{1}a_{1}+\cdots +n_{r}a_{r}\right) t}.
\end{eqnarray*}%
\noindent By applying Mellin transformation to $F_{r,u}\left( x,t\right) $,
we obtain the following integral representation:%
\begin{equation}
\frac{1}{\Gamma \left( s\right) }\int\limits_{0}^{\infty }\frac{%
t^{s-1}e^{-xt}}{\prod_{j=1}^{r}\left( 1-u^{-a_{j}}e^{-a_{j}t}\right) }%
dt=\sum_{n_{1},\ldots ,n_{r}=0}^{\infty }\frac{u^{-\left( n_{1}a_{1}+\cdots
+n_{r}a_{r}\right) }}{\left( x+n_{1}a_{1}+\cdots +n_{r}a_{r}\right) ^{s}}.
\label{Y2}
\end{equation}%
\noindent By (\ref{Y2}), we give the definition of multiple Frobenius-Euler
function

\noindent$l_{r}\left( s,x;u|a_{1},\ldots ,a_{r}\right) $ as follows:

\begin{definition}
\label{def4}For $s\in \mathbb{C}$ with Re$\left( s\right) >r$, we define%
\begin{equation*}
l_{r}\left( s,x;u|a_{1},\ldots ,a_{r}\right) =\sum_{n_{1},\ldots
,n_{r}=0}^{\infty }\frac{u^{-\left( n_{1}a_{1}+\cdots +n_{r}a_{r}\right) }}{%
\left( x+n_{1}a_{1}+\cdots +n_{r}a_{r}\right) ^{s}}
\end{equation*}%
\noindent for Re$\left( x\right) >0$, $a_{1},\ldots ,a_{r}$ positive real
numbers and $u\in \mathbb{C}$, $\left| u\right| \geq 1$.
\end{definition}

\begin{remark}
\label{re1}If we take $r=1$ and $a_{1}=1$ in above definition, we get
Frobenius-Euler $l$-function (\ref{2,4}), and in addition if $x=0$, we get
Frobenius-Euler $l$-function (\ref{2,3}). If $r=1$ and $u=1$, we have
Hurwitz zeta function. If $u=1$, $r=1$ and $x=0$, Riemann zeta function is
obtained (cf. \cite{14}, \cite{24}, \cite{26}, \cite{simjkms2003twqBe}, \cite%
{simBKMStwEulnu}, \cite{simjnttwEul}, \cite{sim2}, \cite{sirivastKimSimsek}).
\end{remark}

\noindent Substituting $s=-n$, $n\in \mathbb{Z}$, $n\geqslant 0$ in (\ref{Y2}%
), by Cauchy residue theorem, we arrive at the following theorem:

\begin{theorem}
\label{th20}For $n\in \mathbb{Z}$, $n\geq 0$, we have%
\begin{equation*}
\prod_{j=1}^{r}\left( 1-u^{-a_{j}}\right) l_{r}\left( -n,x;u|a_{1},\ldots
,a_{r}\right) =H_{r,n}\left( x,u|a_{1},\ldots ,a_{r}\right) .
\end{equation*}
\end{theorem}

We now list some theorems and definitions for the polynomials $H_{n}\left(
x,u\right) $, which are needed in the following sections.

\begin{lemma}
\label{le5}For $n\in \mathbb{Z}$, $n\geq 0$, we have%
\begin{equation*}
H_{n}\left( x+1,u\right) -uH_{n}\left( x,u\right) =\left( 1-u\right) x^{n}.
\end{equation*}
\end{lemma}

By using (\ref{2,1}) and (\ref{2,2}), and after some elementary
calculations, we have%
\begin{eqnarray*}
\sum_{n=0}^{\infty }m^{n}\sum_{j=0}^{m-1}\frac{u^{m-j}}{u^{m}-1}H_{n}\left(
x+\frac{j}{m},u^{m}\right) \frac{t^{n}}{n!} &=&\sum_{j=0}^{m-1}\frac{u^{m-j}%
}{u^{m}-1}\frac{1-u^{m}}{e^{mt}-u^{m}}e^{\left( x+\frac{j}{m}\right) mt} \\
&=&\frac{u}{u-1}\sum_{n=0}^{\infty }H_{n}\left( mx,u\right) \frac{t^{n}}{n!}.
\end{eqnarray*}%
\noindent Therefore, we easily arrive at the following lemma.

\begin{lemma}
\label{le6}For real $x$ and a positive integer $m$,%
\begin{equation*}
m^{n}\sum_{j=0}^{m-1}u^{m-1-j}H_{n}\left( x+\frac{j}{m},u^{m}\right) =\frac{%
u^{m}-1}{u-1}H_{n}\left( mx,u\right)
\end{equation*}%
\noindent for all integers $n\geq 0$.
\end{lemma}

\begin{definition}
\label{def7}(\cite{20}) Let $H_{n}\left( x,u\right) $ denotes the $n$th
Frobenius-Euler polynomial and let $\overline{H}_{n}\left( x,u\right) $ be
defined recursively by%
\begin{equation*}
\overline{H}_{n}\left( x,u\right) =H_{n}\left( x,u\right) \text{, }\left(
0\leq x<1\right) \text{, }\overline{H}_{n}\left( x+1,u\right) =u\overline{H}%
_{n}\left( x,u\right) .
\end{equation*}
\end{definition}

With this definition of $\overline{H}_{n}\left( x,u\right) $, it is easily
verified that Lemma \ref{le6} hold for $\overline{H}_{n}\left( x,u\right) $.

\begin{lemma}
\label{le8}For real $x$ and a positive integer $m$,%
\begin{equation*}
m^{n}\sum_{j=0}^{m-1}u^{m-1-j}\overline{H}_{n}\left( x+\frac{j}{m}%
,u^{m}\right) =\frac{u^{m}-1}{u-1}\overline{H}_{n}\left( mx,u\right)
\end{equation*}%
\noindent for all integers $n\geq 0$.
\end{lemma}

\begin{lemma}
\label{le9}For all integers $n\geq 0$ and $\left( h,k\right) =1$,%
\begin{equation*}
\left( hk\right) ^{n}\sum_{a=0}^{k-1}\sum_{b=0}^{h-1}u^{hk-\left(
kb+ha\right) }\overline{H}_{n}\left( \frac{a}{k}+\frac{b}{h},u^{hk}\right)
=\left( u^{hk}-1\right) \frac{u}{u-1}H_{n}\left( u\right) .
\end{equation*}
\end{lemma}

\begin{proof}
Using Lemma \ref{le8}, the left hand side of the above equation becomes%
\begin{eqnarray*}
\left( hk\right) ^{n}&\sum_{a=0}^{k-1}&\sum_{b=0}^{h-1}u^{hk-\left(
ha+kb\right) }\overline{H}_{n}\left( \frac{a}{k}+\frac{b}{h},u^{hk}\right) \\
&=&\frac{u^{h}}{u^{h}-1}\left( u^{hk}-1\right) h^{n}\sum_{b=0}^{h-1}u^{-kb}%
\overline{H}_{n}\left( \frac{kb}{h},u^{h}\right) .
\end{eqnarray*}%
\noindent For $b=0,1,\ldots ,h-1$, the residues $kb$ mod$h$ are $%
c=0,1,\ldots ,h-1$. Therefore,%
\begin{eqnarray*}
&&\frac{u^{h}}{u^{h}-1}\left( u^{hk}-1\right) h^{n}\sum_{b=0}^{h-1}u^{-kb}%
\overline{H}_{n}\left( \frac{kb}{h},u^{h}\right) \\
&=&\left( u^{hk}-1\right) \frac{u}{u^{h}-1}h^{n}\sum_{c=0}^{h-1}u^{h-c-1}%
\overline{H}_{n}\left( \frac{c}{h},u^{h}\right) =\left( u^{hk}-1\right) 
\frac{u}{u-1}H_{n}\left( u\right) ,
\end{eqnarray*}%
\noindent by Lemma \ref{le8}.
\end{proof}

In the theory of Dedekind sums, the famous relation is reciprocity law,
which plays a major role in this theory and other related topics. We now
give the proof of main theorem for this section, which is related to
reciprocity law for $S_{n,u}\left( h,k\right) $. We use similar methods of
Ota (\cite{8}) and Nagasaka et.al (\cite{9}) for proving Theorem \ref{th11}.

\begin{proof}[Proof of Theorem \protect\ref{th11}]
For $r=2$ and $a_{1}=k$, $a_{2}=h$,%
\begin{eqnarray*}
l_{2}\left( s;u|k,h\right) &=&\sum_{\underset{\left( m,n\right) \neq \left(
0,0\right) }{m,n=0}}^{\infty }\frac{u^{-\left( km+hn\right) }}{\left(
km+hn\right) ^{s}} \\
&=&\sum_{a=0}^{k-1}\sum_{b=0}^{h-1}\sum\limits_{m^{\prime },n^{\prime
}=0}^{\infty }\ \hspace{-0.13in}{^{^{\prime \prime }}}\frac{u^{-\left(
kb+ha+hk\left( m^{\prime }+n^{\prime }\right) \right) }}{\left(
kb+ha+hk\left( m^{\prime }+n^{\prime }\right) \right) ^{s}}
\end{eqnarray*}%
\noindent by writing $n=a+kn^{\prime }$, $m=b+hm^{\prime }$, where $%
\sum^{\prime \prime }$ means that the summation is taken over all positive
integers $m^{\prime }$, $n^{\prime }$ except $\left( m^{\prime },n^{\prime
}\right) =\left( 0,0\right) $ when $a=b=0$. Then for $M=m^{\prime
}+n^{\prime }$,%
\begin{eqnarray}
&l_{2}&\left(
s;u|k,h\right)=\sum_{a=0}^{k-1}\sum_{b=0}^{h-1}\sum\limits_{M=0}^{\infty }\
\ \hspace{-0.13in}{^{^{\prime }}}\frac{\left( M+1\right) u^{-\left(
kb+ha+hkM\right) }}{\left( kb+ha+hkM\right) ^{s}}  \notag \\
&=&\frac{1}{\left( hk\right) ^{s}}\sum_{a=0}^{k-1}\sum_{b=0}^{h-1}u^{-\left(
kb+ha\right) }\sum\limits_{M=0}^{\infty }\ \ \hspace{-0.13in}{^{^{\prime }}}%
\frac{u^{-hkM}}{\left( \frac{kb+ha}{hk}+M\right) ^{s-1}}  \notag \\
&&+\frac{1}{\left( hk\right) ^{s}}\sum_{a=0}^{k-1}\sum_{b=0}^{h-1}u^{-\left(
kb+ha\right) }\sum\limits_{M=0}^{\infty }\ \ \hspace{-0.13in}{^{^{\prime }}}%
\left( 1-\frac{a}{k}-\frac{b}{h}\right) \frac{u^{-hkM}}{\left( \frac{kb+ha}{%
hk}+M\right) ^{s}},  \label{2,5}
\end{eqnarray}%
\noindent where$\sum^{\prime }$ means that the summation is taken over all
positive integers $M$ except $M=0$ when $a=b=0$. By using (\ref{2,4}), we
obtain%
\begin{eqnarray*}
l_{2}\left( s;u|k,h\right) &=&\frac{1}{\left( hk\right) ^{s}}%
\sum_{a=0}^{k-1}\sum_{b=0}^{h-1}u^{-\left( kb+ha\right) }l\left( s-1,\frac{%
kb+ha}{hk};u^{hk}\right) \\
&&+\frac{1}{\left( hk\right) ^{s}}\sum_{a=0}^{k-1}\sum_{b=0}^{h-1}u^{-\left(
kb+ha\right) }\left( 1-\frac{a}{k}-\frac{b}{h}\right) l\left( s,\frac{kb+ha}{%
hk};u^{hk}\right) .
\end{eqnarray*}%
\noindent By substituting $s=-n$, $n\in \mathbb{Z}$, $n\geq 0$, into (\ref%
{2,5}) and using (\ref{2,9}), we have%
\begin{eqnarray*}
&l_{2}&\left( -n,u|k,h\right)=\frac{\left( hk\right) ^{n}}{u^{hk}-1}%
\sum_{a=0}^{k-1}\sum_{b=0}^{h-1}u^{hk-\left( kb+ha\right) }H_{n+1}\left( 
\frac{a}{k}+\frac{b}{h},u^{hk}\right) \\
&&+\frac{\left( hk\right) ^{n}}{u^{hk}-1}\sum_{a=0}^{k-1}%
\sum_{b=0}^{h-1}u^{hk-\left( kb+ha\right) }\left( 1-\frac{a}{k}-\frac{b}{h}%
\right) H_{n}\left( \frac{a}{k}+\frac{b}{h},u^{hk}\right) ,
\end{eqnarray*}%
\noindent where the values $a$ and $b$ in the sums satisfy%
\begin{equation*}
0\leq \frac{a}{k}+\frac{b}{h}<2\text{ and }\frac{a}{k}+\frac{b}{h}\neq 1.
\end{equation*}%
\noindent Let $B$ be the set defined by%
\begin{equation}
B=\left\{ \left( a,b\right) \in \mathbb{Z}\times \mathbb{Z}:0\leq a\leq k-1%
\text{, }0\leq b\leq h-1\text{, }\frac{a}{k}+\frac{b}{h}>1\right\} .
\label{3,5}
\end{equation}%
\noindent Then by Lemma \ref{le5} and Definition \ref{def7}, we obtain%
\begin{eqnarray*}
&l_{2}&\left( -n;u|k,h\right)=\frac{\left( hk\right) ^{n}}{u^{hk}-1}%
\sum_{a=0}^{k-1}\sum_{b=0}^{h-1}u^{hk-\left( kb+ha\right) }\overline{H}%
_{n+1}\left( \frac{a}{k}+\frac{b}{h},u^{hk}\right) \\
&&+\frac{\left( hk\right) ^{n}}{u^{hk}-1}\left( 1-u^{hk}\right) \sum_{\left(
a,b\right) \in B}u^{hk-\left( kb+ha\right) }\left( \frac{a}{k}+\frac{b}{h}%
-1\right) ^{n+1} \\
&&+\frac{\left( hk\right) ^{n}}{u^{hk}-1}\sum_{a=0}^{k-1}%
\sum_{b=0}^{h-1}u^{hk-\left( kb+ha\right) }\left( 1-\frac{a}{k}-\frac{b}{h}%
\right) \overline{H}_{n}\left( \frac{a}{k}+\frac{b}{h},u^{hk}\right) \\
&&+\frac{\left( hk\right) ^{n}}{u^{hk}-1}\left( 1-u^{hk}\right) \sum_{\left(
a,b\right) \in B}u^{hk-\left( kb+ha\right) }\left( 1-\frac{a}{k}-\frac{b}{h}%
\right) \left( \frac{a}{k}+\frac{b}{h}-1\right) ^{n}.
\end{eqnarray*}%
\noindent Now, by using Lemma \ref{le8}, Lemma \ref{le9} and Definition \ref%
{def10}, we get%
\begin{eqnarray}
&l_{2}&\left( -n;u|k,h\right)=\frac{\left( hk\right) ^{n}}{u^{hk}-1}%
\sum_{a=0}^{k-1}\sum_{b=0}^{h-1}u^{hk-\left( kb+ha\right) }\overline{H}%
_{n+1}\left( \frac{a}{k}+\frac{b}{h},u^{hk}\right)  \notag \\
&&+\frac{\left( hk\right) ^{n}}{u^{hk}-1}\sum_{a=0}^{k-1}%
\sum_{b=0}^{h-1}u^{hk-\left( kb+ha\right) }\overline{H}_{n}\left( \frac{a}{k}%
+\frac{b}{h},u^{hk}\right)  \notag \\
&&-\frac{\left( hk\right) ^{n}}{u^{hk}-1}\sum_{a=0}^{k-1}%
\sum_{b=0}^{h-1}u^{hk-\left( kb+ha\right) }\frac{a}{k}\overline{H}_{n}\left( 
\frac{a}{k}+\frac{b}{h},u^{hk}\right)  \notag \\
&&-\frac{\left( hk\right) ^{n}}{u^{hk}-1}\sum_{a=0}^{k-1}%
\sum_{b=0}^{h-1}u^{hk-\left( kb+ha\right) }\frac{b}{h}\overline{H}_{n}\left( 
\frac{a}{k}+\frac{b}{h},u^{hk}\right)  \notag \\
&=&\frac{1}{hk}\frac{u}{u-1}H_{n+1}\left( u\right) +\frac{u}{u-1}H_{n}\left(
u\right)  \notag \\
&&-\frac{u^{k}}{u^{k}-1}k^{n}S_{n,u^{k}}\left( h,k\right) -\frac{u^{h}}{%
u^{h}-1}h^{n}S_{n,u^{h}}\left( k,h\right) .  \label{2,6}
\end{eqnarray}

By definition of $l_{2}$ $\left( s;u|k,h\right) $, we have%
\begin{eqnarray}
&l_{2}&\left( -n;u|k,h\right)=\frac{H_{2,n}\left( u|k,h\right) }{\left(
1-u^{-k}\right) \left( 1-u^{-h}\right) }=\frac{\left( ^{1}H\left(
u^{k}\right) k+^{2}H\left( u^{h}\right) h\right) ^{n}}{\left(
1-u^{-k}\right) \left( 1-u^{-h}\right) }  \notag \\
&=&\frac{u^{k}u^{h}}{\left( u^{k}-1\right) \left( u^{h}-1\right) }%
\sum_{j=0}^{n}\binom{n}{j}H_{j}\left( u^{k}\right) H_{n-j}\left(
u^{h}\right) k^{j}h^{n-j}.  \label{2,7}
\end{eqnarray}%
\noindent By (\ref{2,6}) and (\ref{2,7}), we have%
\begin{eqnarray*}
&&\left( \frac{u^{k}}{1-u^{k}}k^{n}S_{n,u^{k}}\left( h,k\right) +\frac{u^{h}%
}{1-u^{h}}h^{n}S_{n,u^{h}}\left( k,h\right) \right) \\
&=&\sum_{j=0}^{n}\binom{n}{j}\frac{u^{k}}{1-u^{k}}H_{j}\left( u^{k}\right)
k^{j}\frac{u^{h}}{1-u^{h}}H_{n-j}\left( u^{h}\right) h^{n-j} \\
&&+\frac{1}{hk}\frac{u}{1-u}H_{n+1}\left( u\right) +\frac{u}{1-u}H_{n}\left(
u\right) .
\end{eqnarray*}%
\noindent Thus, we arrive the desired result.
\end{proof}

\section{Generalized Dedekind Type Sums Attached to a Dirichlet Character}

\setcounter{equation}{0} \setcounter{theorem}{0}

Character generalizations of classical Dedekind sums have been studied by
many mathematicians. By using generalized Bernoulli functions attached to
character, Berndt \cite{5} defined Dedekind sums with characters for $n=1$,
and proved reciprocity laws by using either Eisenstein series with
characters (\cite{5}, \cite{33}), integrals such as contour integrals and
Riemann-Stieltjes integrals, or the Poisson summation formula (\cite{34}).
Nagasaka et.al \cite{9} defined generalized character Dedekind sums which
are different from Berndt's definitions for the case $n=1$, and Cenkci et.al 
\cite{35} gave exact generalizations of Berndt's sums to the case any
positive number. Simsek \cite{simjnt2003dede}, \cite{simADCMgenDedeEisn}, 
\cite{simjmaaqDed}, \cite{simASCMsyang} considered Dedekind sums. He gave
several properties of these sums. In \cite{37}, Zhang studied the
distribution property of a sum analogous to the Dedekind sums by using mean
value theorem of the Dirichlet $L$-function. Xiali and Zhang \cite{38}
studied the asymptotic behavior of the Dedekind sums with a weight of
Hurwitz zeta function by applying the mean value theorem of the Dirichlet $L$%
-function.

To prove Theorem \ref{th4}, we need the following definitions.

\begin{definition}
\label{def13}(\cite{21}, \cite{sirivastKimSimsek}, \cite{32}) For a
primitive Dirichlet character $\chi $ of conductor $f=f_{\chi }$,
generalized Frobenius-Euler numbers attached to $\chi $, $H_{n,\chi }\left(
u\right) $, $n\in \mathbb{Z}$, $n\geq 0$, are defined by means of%
\begin{equation}
\sum_{n=0}^{\infty }H_{n,\chi }\left( u\right) \frac{t^{n}}{n!}%
=\sum_{a=0}^{f-1}\frac{\chi \left( a\right) \left( 1-u^{f}\right)
u^{f-a}e^{at}}{e^{ft}-u^{f}}.  \label{2,8}
\end{equation}%
\noindent With this definition, it is easy to verify that%
\begin{equation*}
H_{n,\chi }\left( u\right) =f^{n}\sum_{a=0}^{f-1}\chi \left( a\right)
u^{f-a}H_{n}\left( \frac{a}{f},u^{f}\right) .
\end{equation*}%
\noindent Also, if $F$ is an integer multiple of $f$, we have%
\begin{equation}
\frac{1-u^{F}}{u^{f}-1}u^{f}H_{n,\chi }\left( u\right)
=F^{n}\sum_{a=0}^{F-1}\chi \left( a\right) u^{F-a}H_{n}\left( \frac{a}{F}%
,u^{F}\right) .  \label{MU1}
\end{equation}
\end{definition}

\begin{definition}
\label{def14}Let $\chi $ be a Dirichlet character of conductor $f=f_{\chi }$
with $f|hk$. We define the double $l$-function $l_{2}\left( s;u;\chi
|k,h\right) $ with parameters $\left( k,h\right) $, $u$ and $\chi $ by%
\begin{equation}
l_{2}\left( s;u;\chi |k,h\right) =\sum_{\underset{\left( n,m\right) \neq
\left( 0,0\right) }{n,m=0}}^{\infty }\frac{\chi \left( km+hn\right)
u^{-\left( km+hn\right) }}{\left( km+hn\right) ^{s}}  \label{3,3}
\end{equation}%
\noindent for $s\in \mathbb{C}$, Re$\left( s\right) >2$, $u\in \mathbb{C}$, $%
\left| u\right| \geq 1$.
\end{definition}

We observe that for $u=1$, (\ref{3,3}) reduces to double zeta function%
\begin{equation*}
\widetilde{\zeta }_{2}\left( s;\left( k,h\right) ,\chi \right) =\sum_{%
\underset{\left( n,m\right) \neq \left( 0,0\right) }{n,m=0}}^{\infty }\frac{%
\chi \left( km+hn\right) }{\left( km+hn\right) ^{s}},
\end{equation*}

\noindent defined in \cite{9}. Also for primitive character $\chi =1$, (\ref%
{3,3}) reduces to the double $l$-function $l_{2}\left( s;u|k,h\right) $
defined in Section 2.

$l_{2}\left( s;u;\chi |k,h\right) $ can be analytically continued to the
whole plane by the following identities:%
\begin{eqnarray}
l_{2}\left( s;u;\chi |k,h\right)
&=&\sum_{a=0}^{k-1}\sum_{b=0}^{h-1}\sum\limits_{m^{\prime },n^{\prime
}=0}^{\infty }\ \hspace{-0.13in}{^{^{\prime \prime }}}\frac{\chi \left(
kb+ha\right) u^{-\left( kb+ha+hk\left( m^{\prime }+n^{\prime }\right)
\right) }}{\left( kb+ha+hk\left( m^{\prime }+n^{\prime }\right) \right) ^{s}}
\notag \\
&=&\sum_{a=0}^{k-1}\sum_{b=0}^{h-1}\chi \left( kb+ha\right) u^{-\left(
kb+ha\right) }l_{2}\left( s,kb+ha;u|hk,hk\right) .  \label{1,2}
\end{eqnarray}

Now we give proof of Theorem \ref{th4} as follows:

\begin{proof}[Proof of Theorem \protect\ref{th4}]
By substituting $s=-n$, $n\in \mathbb{Z}$, $n\geq 0$ in (\ref{1,2}), we have 
\begin{eqnarray}
&l_{2}&\left( -n;u;\chi |k,h\right)=\sum_{a=0}^{k-1}\sum_{b=0}^{h-1}\chi
\left( kb+ha\right) u^{-\left( kb+ha\right) }l_{2}\left(
-n,kb+ha;u|hk,hk\right)  \notag \\
&=&\sum_{a=0}^{k-1}\sum_{b=0}^{h-1}\chi \left( kb+ha\right) u^{-\left(
kb+ha\right) }\frac{H_{2,n}\left( kb+ha,u|hk,hk\right) }{\left(
1-u^{-hk}\right) ^{2}}  \notag \\
&=&\sum_{a=0}^{k-1}\sum_{b=0}^{h-1}\chi \left( kb+ha\right) u^{-\left(
kb+ha\right) }  \notag \\
&&\times \frac{\left( ^{1}H\left( u^{hk}\right) hk+\text{ }^{2}H\left(
u^{hk}\right) hk+kb+ha\right) ^{n}}{\left( 1-u^{-hk}\right) ^{2}}.
\label{3,4}
\end{eqnarray}%
\noindent By substituting $m^{\prime }+n^{\prime }=M$ in (\ref{1,2}), we get%
\begin{eqnarray*}
&l_{2}&\left( s;u;\chi |k,h\right)=\frac{1}{\left( hk\right) ^{s}}%
\sum_{a=0}^{k-1}\sum_{b=0}^{h-1}\chi \left( kb+ha\right) u^{-\left(
kb+ha\right) }l\left( s-1,\frac{kb+ha}{hk};u^{hk}\right) \\
&&+\frac{1}{\left( hk\right) ^{s}}\sum_{a=0}^{k-1}\sum_{b=0}^{h-1}\chi
\left( kb+ha\right) u^{-\left( kb+ha\right) }\left( 1-\frac{a}{k}-\frac{b}{h}%
\right) l\left( s,\frac{kb+ha}{hk};u^{hk}\right) .
\end{eqnarray*}%
\noindent Setting $s=-n$, $n\in \mathbb{Z}$, $n\geq 0$ in the above equality
yields%
\begin{eqnarray*}
&l_{2}&\left( -n;u;\chi |k,h\right) \\
&&=\frac{\left( hk\right) ^{n}}{u^{hk}-1}\sum_{a=0}^{k-1}\sum_{b=0}^{h-1}%
\chi \left( kb+ha\right) u^{hk-\left( kb+ha\right) }H_{n+1}\left( \frac{a}{k}%
+\frac{b}{h},u^{hk}\right) \\
&&+\frac{\left( hk\right) ^{n}}{u^{hk}-1}\sum_{a=0}^{k-1}\sum_{b=0}^{h-1}%
\chi \left( kb+ha\right) u^{hk-\left( kb+ha\right) }\left( 1-\frac{a}{k}-%
\frac{b}{h}\right) H_{n}\left( \frac{a}{k}+\frac{b}{h},u^{hk}\right)
\end{eqnarray*}
\begin{eqnarray}
&&=\frac{\left( hk\right) ^{n}}{u^{hk}-1}\sum_{a=0}^{k-1}\sum_{b=0}^{h-1}%
\chi \left( kb+ha\right) u^{hk-\left( kb+ha\right) }\overline{H}_{n+1}\left( 
\frac{a}{k}+\frac{b}{h},u^{hk}\right)  \notag \\
&&+\frac{\left( hk\right) ^{n}}{u^{hk}-1}\left( 1-u^{hk}\right) \sum_{\left(
a,b\right) \in B}\chi \left( kb+ha\right) u^{hk-\left( kb+ha\right) }\left( 
\frac{a}{k}+\frac{b}{h}-1\right) ^{n+1}  \notag \\
&&+\frac{\left( hk\right) ^{n}}{u^{hk}-1}\sum_{a=0}^{k-1}\sum_{b=0}^{h-1}%
\chi \left( kb+ha\right) u^{hk-\left( kb+ha\right) }\left( 1-\frac{a}{k}-%
\frac{b}{h}\right) \overline{H}_{n}\left( \frac{a}{k}+\frac{b}{h}%
,u^{hk}\right)  \notag \\
&&+\frac{\left( hk\right) ^{n}}{u^{hk}-1}\left( 1-u^{hk}\right) \sum_{\left(
a,b\right) \in B}\chi \left( kb+ha\right) u^{hk-\left( kb+ha\right) }  \notag
\\
&&\times \left( 1-\frac{a}{k}-\frac{b}{h}\right) \left( \frac{a}{k}+\frac{b}{%
h}-1\right) ^{n},  \label{3,7}
\end{eqnarray}%
\noindent where $B$ is defined by (\ref{3,5}). From (\ref{MU1}), we have%
\begin{eqnarray}
\left( hk\right) ^{n}\sum_{a=0}^{k-1}\sum_{b=0}^{h-1}&&\chi \left(
kb+ha\right) u^{hk-\left( kb+ha\right) }\overline{H}_{n}\left( \frac{a}{k}+%
\frac{b}{h},u^{hk}\right)  \notag \\
&=&\frac{1-u^{hk}}{u^{f}-1}u^{f}H_{n,\chi }\left( u\right) ,  \label{MU2}
\end{eqnarray}%
\noindent since for the values $a$ and $b$ in the sums, we have%
\begin{equation*}
\left\{ \frac{kb+ha}{hk}-\left[ \frac{kb+ha}{hk}\right] _{G}:0\leqslant a<k%
\text{, }0\leqslant b<h\right\} =\left\{ \frac{i}{hk}:0\leqslant
i<hk\right\} .
\end{equation*}%
\noindent By using (\ref{MU2}) in (\ref{3,7}), we obtain%
\begin{eqnarray*}
l_{2}\left( -n;u;\chi |k,h\right) &=&\frac{u^{f}}{1-u^{f}}\frac{1}{hk}%
H_{n+1,\chi }\left( u\right) +\frac{u^{f}}{1-u^{f}}H_{n,\chi }\left( u\right)
\\
&&-\frac{u^{hk}}{u^{hk}-1}\left( k^{n}S_{n,u^{k}}\left( h,k|\chi \right)
+h^{n}S_{n,u^{h}}\left( k,h|\chi \right) \right) ,
\end{eqnarray*}%
\noindent which together with (\ref{3,4}) completes the proof.
\end{proof}

\section{Twisted Version of Dedekind Type Sums}

\setcounter{equation}{0} \setcounter{theorem}{0}

One of the curious facts about Frobenius-Euler polynomials is the
relationship between Bernoulli polynomials. This relationship occurs when $%
u=\zeta $ is any root of unity. For example, for the generating function of
Bernoulli polynomials $B_{n}\left( x\right) $, $n\in \mathbb{Z}$, $n\geq 0$,%
\begin{equation*}
\frac{te^{xt}}{e^{t}-1}=\sum_{n=0}^{\infty }B_{n}\left( x\right) \frac{t^{n}%
}{n!},
\end{equation*}%
\noindent we have%
\begin{equation*}
\sum_{n=0}^{\infty }\frac{t^{n}}{n!}\sum_{j=0}^{m-1}\zeta ^{-rj}B_{n}\left(
x+\frac{j}{m}\right) =\frac{te^{xt}}{e^{t}-1}\sum_{j=0}^{m-1}\zeta
^{-rj}e^{jt/m}=\frac{te^{xt}}{\zeta ^{-r}e^{t/m}-1}=\frac{\zeta ^{r}te^{xt}}{%
e^{t/m}-\zeta ^{r}},
\end{equation*}%
\noindent where $\zeta $ is a primitive $m$th root of unity and $m\not| r$.
This implies%
\begin{equation}
m^{n-1}\sum_{j=0}^{m-1}\zeta ^{-rj}B_{n}\left( x+\frac{j}{m}\right) =\frac{%
n\zeta ^{r}}{1-\zeta ^{r}}H_{n-1}\left( mx,\zeta ^{r}\right) .  \label{1,3}
\end{equation}%
\noindent For $m|r$, we have (\ref{Y2*}). Furthermore, multiplying both
sides of (\ref{1,3}) by $\zeta ^{rj}$, summing over $r$ and (\ref{Y2*}), we
get%
\begin{equation*}
m^{n}B_{n}\left( x+\frac{j}{m}\right) =B_{n}\left( mx\right)
+n\sum_{r=1}^{m-1}\zeta ^{rj}\frac{H_{n-1}\left( mx,\zeta ^{r}\right) }{%
\zeta ^{-r}-1},
\end{equation*}%
\noindent where $0\leq j<m$ (\cite{20}).

We define twisted Dedekind type sums as follows.

\begin{definition}
\label{def15}Let $n$ be a positive integer, $h$, $k$ be relatively prime
positive integers and $\zeta ^{hk-1}=1$, $\zeta \neq 1$. Then, we define
twisted Dedekind sums by%
\begin{equation*}
S_{n,\zeta }\left( h,k\right) =\sum_{a=0}^{k-1}\zeta ^{-\frac{ha}{k}}\frac{a%
}{k}\overline{H}_{n}\left( \frac{ha}{k},\zeta \right) .
\end{equation*}
\end{definition}

\begin{definition}
\label{def17}Let $\chi $ be a Dirichlet character of conductor $f=f_{\chi }$
with $f|hk$, $n$ be a positive integer and $\zeta ^{hk-1}=1$, $\zeta \neq 1$%
. Then we define%
\begin{equation*}
S_{n,\zeta ^{k}}\left( h,k|\chi \right)
=h^{n}\sum_{a=0}^{k-1}\sum_{b=0}^{h-1}\chi \left( kb+ha\right) \zeta
^{-\left( kb+ha\right) }\frac{a}{k}\overline{H}_{n}\left( \frac{a}{k}+\frac{b%
}{h},\zeta \right) .
\end{equation*}
\end{definition}

Observe that for $\zeta =u$ Definition \ref{def15} and Definition \ref{def17}
reduce to Definition \ref{def10} and Definition \ref{def19}, respectively.

By substituting $u=\zeta $ in Theorem \ref{th11} and Theorem \ref{th4} with $%
\zeta ^{hk-1}=1$, $\zeta \neq 1$, we obtain the reciprocity laws for $%
S_{n,\zeta }\left( h,k\right) $ and $S_{n,\zeta ^{k}}\left( h,k|\chi \right) 
$ as follows:

\begin{theorem}
\label{th16}For positive integer $n$, relatively prime positive integers $h$%
, $k$ and $\zeta ^{hk-1}=1$, $\zeta \neq 1$, we have%
\begin{eqnarray*}
\frac{\zeta ^{k}}{1-\zeta ^{k}}&&k^{n}S_{n,\zeta ^{k}}\left( h,k\right) +%
\frac{\zeta ^{h}}{1-\zeta ^{h}}h^{n}S_{n,\zeta ^{h}}\left( k,h\right) \\
&&=\sum_{j=0}^{n}\binom{n}{j}\frac{\zeta ^{k}}{1-\zeta ^{k}}H_{j}\left(
\zeta ^{k}\right) k^{j}\frac{\zeta ^{h}}{1-\zeta ^{h}}H_{n-j}\left( \zeta
^{h}\right) h^{n-j} \\
&&+\frac{1}{hk}\frac{1}{1-\zeta }H_{n+1}\left( \zeta \right) +\frac{1}{%
1-\zeta }H_{n}\left( \zeta \right) .
\end{eqnarray*}
\end{theorem}

\begin{theorem}
\label{th18}Let $n$, $\chi $ and $\zeta $ be as in Definition \ref{def17}.
Then%
\begin{eqnarray*}
k^{n}S_{n,\zeta ^{k}}&&\left( h,k|\chi \right) +h^{n}S_{n,\zeta ^{h}}\left(
k,h|\chi \right) \\
&&=\frac{1-\zeta }{1-\zeta ^{f}}\zeta ^{f}\left( \frac{\zeta ^{-1}}{hk}%
H_{n+1,\chi }\left( u\right) +\zeta ^{-1}H_{n,\chi }\left( u\right) \right)
\\
&&+\frac{\zeta }{\zeta -1}\sum_{a=0}^{k-1}\sum_{b=0}^{h-1}\chi \left(
kb+ha\right) \zeta ^{-\left( kb+ha\right) } \\
&&\times \left( ^{1}H\left( \zeta \right) hk+\text{ }^{2}H\left( \zeta
\right) hk+kb+ha\right) ^{n}. \\
\end{eqnarray*}
\end{theorem}

\section{$\left( h,q\right) $-Approach to $p$-adic Twisted Dedekind Type Sums%
}

\setcounter{equation}{0} \setcounter{theorem}{0}

In this section, we define twisted $\left( h,q\right) $-Dedekind type sums
by using twisted $\left( h,q\right) $-Bernoulli polynomials. By using $p$%
-adic interpolation of certain partial zeta function, we interpolate these
sums to construct twisted $\left( h,q\right) $-$p$-adic Dedekind sums.

>From Definition \ref{def12} and binomial expansion, we have%
\begin{eqnarray}
s_{m,\zeta }^{\left( h\right) }\left( a,b:q\right) &=&\sum_{j=0}^{b-1}\frac{j%
}{b}\int\limits_{\mathbb{Z}_{p}}q^{hx}\zeta ^{x}\left( x+\left\{ \frac{ja}{b}%
\right\} \right) ^{m}d\mu _{1}\left( x\right)  \notag \\
&=&\sum_{j=0}^{b-1}\sum_{c=0}^{m}\frac{j}{b}\binom{m}{c}\left\{ \frac{ja}{b}%
\right\} ^{m-c}\int\limits_{\mathbb{Z}_{p}}q^{hx}\zeta ^{x}x^{c}d\mu
_{1}\left( x\right)  \label{Y6}
\end{eqnarray}%
\noindent By using (\ref{Y4}), we obtain%
\begin{equation}
B_{m,\zeta }^{\left( h\right) }\left( x,q\right) =\sum_{c=0}^{m}\binom{m}{c}%
B_{c,\zeta }^{\left( h\right) }\left( q\right) x^{m-c}.  \label{Y5}
\end{equation}%
\noindent By substituting (\ref{Y5}) into (\ref{Y6}), we get%
\begin{equation*}
s_{m,\zeta }^{\left( h\right) }\left(
a,b:q\right)=\sum_{j=0}^{b-1}\sum_{c=0}^{m}\frac{j}{b}\binom{m}{c}B_{c,\zeta
}^{\left( h\right) }\left( q\right)\left\{ \frac{ja}{b}\right\}
^{m-c}=\sum_{j=0}^{b-1}\frac{j}{b}B_{m,\zeta }^{\left( h\right) }\left(
\left\{ \frac{ja}{b}\right\} ,q\right) .
\end{equation*}

Throughout this section, $\omega $ will denote the Teichm\"{u}ller character 
$\left( \text{mod}p\right) $ and we assume that $q\in \mathbb{C}_{p}$ with $%
\left| 1-q\right| _{p}<p^{-1/\left( p-1\right) }$, and $\zeta \in T_{p}$.

\begin{theorem}
\label{th13}Let $a$, $b$, $p$ and $\zeta $ be as in Definition \ref{def12}.
Then there exists a $p$-adic continuous function $S_{p,\zeta }^{\left(
h\right) }\left( s;a,b:q\right) $ of $s$ on $\mathbb{Z}_{p}$ which satisfies%
\begin{equation*}
S_{p,\zeta }^{\left( h\right) }\left( m;a,b:q\right) =b^{m}s_{m,\zeta
}^{\left( h\right) }\left( a,b:q\right)
\end{equation*}%
\noindent for all positive integers $m$ such that $m+1\equiv 0\left( \text{%
mod}\left( p-1\right) \right) $.
\end{theorem}

\begin{proof}
Proof of this theorem is similar to that of Theorem 5 of \cite{19} and
Theorem 7 of \cite{simjkms2006}. Let $p$ be an odd prime, $j$ and $b$
positive integers such that $\left( p,j\right) =1$ and $p|b$. Then, we define%
\begin{equation}
T_{\zeta }^{\left( h\right) }\left( s;j,b:q\right) =\omega ^{-1}\left(
j\right) \frac{\left\langle j\right\rangle ^{s}}{b}\sum_{k=0}^{\infty }%
\binom{s}{k}\left( \frac{b}{j}\right) ^{k}B_{k,\zeta }^{\left( h\right)
}\left( q\right) ,  \label{3,1}
\end{equation}%
\noindent for $s\in \mathbb{Z}_{p}$, where $\left\langle x\right\rangle
=x\omega ^{-1}\left( x\right) $. Since%
\begin{equation*}
\left| \binom{s}{k}\right| _{p}\leq 1\text{, }\left| \frac{b}{j}\right|
_{p}<1\text{ and }\left| B_{k,\zeta }^{\left( h\right) }\left( q\right)
\right| _{p}\leq 1,
\end{equation*}%
\noindent the sum%
\begin{equation*}
\sum_{k=0}^{\infty }\binom{s}{k}\left( \frac{b}{j}\right) ^{k}B_{k,\zeta
}^{\left( h\right) }\left( q\right)
\end{equation*}%
\noindent converges to a continuous function of $s$ in $\mathbb{Z}_{p}$.

Substituting $s=m$ in (\ref{3,1}), we have%
\begin{eqnarray*}
T_{\zeta }^{\left( h\right) }\left( m;j,b:q\right) &=&\omega ^{-1}\left(
j\right) \frac{\left\langle j\right\rangle ^{m}}{b}\sum_{k=0}^{m}\binom{m}{k}%
\left( \frac{b}{j}\right) ^{k}B_{k,\zeta }^{\left( h\right) }\left( q\right)
\\
&=&\omega ^{-m-1}\left( j\right) b^{m-1}\sum_{k=0}^{m}\binom{m}{k}\left( 
\frac{j}{b}\right) ^{m-k}B_{k,\zeta }^{\left( h\right) }\left( q\right) \\
&=&\omega ^{-m-1}\left( j\right) b^{m-1}B_{m,\zeta }^{\left( h\right)
}\left( \frac{j}{b},q\right) .
\end{eqnarray*}%
\noindent If $m+1\equiv 0\left( \text{mod}\left( p-1\right) \right) $, then%
\begin{equation*}
T_{\zeta }^{\left( h\right) }\left( m;j,b:q\right) =b^{m-1}B_{m,\zeta
}^{\left( h\right) }\left( \frac{j}{b},q\right) .
\end{equation*}%
\noindent Consequently, $T_{\zeta }^{\left( h\right) }\left( m;j,b:q\right) $
is continuous $p$-adic extension of $b^{m-1}B_{m,\zeta }^{\left( h\right)
}\left( \frac{j}{b},q\right) $.

Now, since%
\begin{equation*}
s_{m,\zeta }^{\left( h\right) }\left( a,b:q\right) =\sum_{j=0}^{b-1}\frac{j}{%
b}B_{m,\zeta }^{\left( h\right) }\left( \left\{ \frac{aj}{b}\right\}
,q\right)
\end{equation*}%
\noindent and%
\begin{equation*}
T_{\zeta }^{\left( h\right) }\left( m;j,b:q\right) =b^{m-1}B_{m,\zeta
}^{\left( h\right) }\left( \frac{j}{b},q\right) ,
\end{equation*}%
\noindent we have%
\begin{equation*}
b^{m}s_{m,\zeta }^{\left( h\right) }\left( a,b:q\right)
=\sum_{j=0}^{b-1}jT_{\zeta }^{\left( h\right) }\left( m;\left( aj\right)
_{b},b:q\right)
\end{equation*}%
\noindent for $p|b$ and $m+1\equiv 0\left( \text{mod}\left( p-1\right)
\right) $.
\end{proof}

In the sequel, we construct twisted $\left( h,q\right) $-character Dedekind
type sums. These sums are new and generalize the sums defined by Kudo \cite%
{12}, \cite{11}, Rosen and Synder \cite{10} and Kim \cite{19}.

Generalization of Definition \ref{def12} is given by the following
definition.

\begin{definition}
\label{def18}Let $a$, $b$ be fixed integers with $\left( a,b\right) =1$, and
let $p$ be an odd prime such that $p|b$. For a primitive Dirichlet character
with conductor $f=f_{\chi }$ and $\zeta \in T_{p}$, we define twisted $%
\left( h,q\right) $-Dedekind sums as%
\begin{equation*}
s_{m,\zeta }^{\left( h\right) }\left( a,b:q,\chi \right)
=\sum_{j=0}^{fb-1}\chi \left( j\right) \frac{j}{fb}B_{m,\zeta ,\chi
}^{\left( h\right) }\left( \left\{ \frac{aj}{b}\right\} ,q\right) .
\end{equation*}
\end{definition}

We now generalize Theorem \ref{th13} by character $\chi $. Observe that when 
$\chi =\chi _{0}$, the principle character, Definition \ref{def18} reduces
to Definition \ref{def12}, and the following theorem reduces to Theorem \ref%
{th13}.

\begin{theorem}
\label{th19}Let $a$, $b$, $p$, $\chi $ and $\zeta $ be as in Definition \ref%
{def18}. Then there exists a $p$-adic continuous function $S_{p,\zeta ,\chi
}^{\left( h\right) }\left( s;a,b:q\right) $ of $s$ on $\mathbb{X}$ which
satisfies%
\begin{equation*}
S_{p,\zeta ,\chi }^{\left( h\right) }\left( m;a,b:q\right) =fb^{m}s_{m,\zeta
}^{\left( h\right) }\left( a,b:q,\chi \right)
\end{equation*}%
\noindent for all positive integers $m$ such that $m+1\equiv 0\left( \text{%
mod}\left( p-1\right) \right) $.
\end{theorem}

\begin{proof}
We follow the similar method in the proof of Theorem \ref{th13}. Let $p$ be
an odd prime, $j$ and $b$ positive integers such that $\left( p,j\right) =1$
and $p|b$. For an embedding of the algebraic closure of $\mathbb{Q}$, $%
\overline{\mathbb{Q}}$, into $\mathbb{C}_{p}$, we may consider the values of
a Dirichlet character $\chi $ as lying in $\mathbb{C}_{p}$. Then we define%
\begin{equation}
T_{\zeta ,\chi }^{\left( h\right) }\left( s;j,b:q\right) =\omega ^{-1}\left(
j\right) \frac{\left\langle j\right\rangle ^{s}}{b}\sum_{k=0}^{\infty }%
\binom{s}{k}\left( \frac{b}{j}\right) ^{k}B_{k,\zeta ,\chi }^{\left(
h\right) }\left( q\right) ,  \label{3,2}
\end{equation}%
\noindent for $s\in \mathbb{X}$. Since%
\begin{equation*}
\left| \binom{s}{k}\right| _{p}\leq 1\text{, }\left| \frac{b}{j}\right|
_{p}<1\text{ and }\left| B_{k,\zeta ,\chi }^{\left( h\right) }\left(
q\right) \right| _{p}\leq 1,
\end{equation*}%
\begin{equation*}
\sum_{k=0}^{\infty }\binom{s}{k}\left( \frac{b}{j}\right) ^{k}B_{k,\zeta
,\chi }^{\left( h\right) }\left( q\right)
\end{equation*}%
\noindent converges to a continuous function of $s$ in $\mathbb{X}$.

Substituting $s=m$ in (\ref{3,2}), we have%
\begin{eqnarray*}
T_{\zeta ,\chi }^{\left( h\right) }\left( m;j,b:q\right) &=&\omega
^{-1}\left( j\right) \frac{\left\langle j\right\rangle ^{m}}{b}\sum_{k=0}^{m}%
\binom{m}{k}\left( \frac{b}{j}\right) ^{k}B_{k,\zeta ,\chi }^{\left(
h\right) }\left( q\right) \\
&=&\omega ^{-m-1}\left( j\right) b^{m-1}\sum_{k=0}^{m}\binom{m}{k}\left( 
\frac{j}{b}\right) ^{m-k}B_{k,\zeta ,\chi }^{\left( h\right) }\left( q\right)
\\
&=&\omega ^{-m-1}\left( j\right) b^{m-1}B_{m,\zeta ,\chi }^{\left( h\right)
}\left( \frac{j}{b},q\right) .
\end{eqnarray*}%
\noindent If $m+1\equiv 0\left( \text{mod}\left( p-1\right) \right) $, then%
\begin{equation*}
T_{\zeta ,\chi }^{\left( h\right) }\left( m;j,b:q\right) =b^{m-1}B_{m,\zeta
,\chi }^{\left( h\right) }\left( \frac{j}{b},q\right) .
\end{equation*}%
\noindent Consequently, $T_{\zeta ,\chi }^{\left( h\right) }\left(
m,j,b:q\right) $ is continuous $p$-adic extension of $b^{m-1}B_{m,\zeta
,\chi }^{\left( h\right) }\left( \frac{j}{b},q\right) $.

Now, since%
\begin{equation*}
s_{m,\zeta }^{\left( h\right) }\left( a,b:q,\chi \right)
=\sum_{j=0}^{fb-1}\chi \left( j\right) \frac{j}{fb}B_{m,\zeta ,\chi
}^{\left( h\right) }\left( \left\{ \frac{aj}{b}\right\} ,q\right)
\end{equation*}%
\noindent and%
\begin{equation*}
T_{\zeta ,\chi }^{\left( h\right) }\left( m;j,b:q\right) =b^{m-1}B_{m,\zeta
,\chi }^{\left( h\right) }\left( \frac{j}{b},q\right) ,
\end{equation*}%
\noindent we have%
\begin{equation*}
fb^{m}s_{m,\zeta }^{\left( h\right) }\left( a,b:q,\chi \right)
=\sum_{j=0}^{fb-1}j\chi \left( j\right) T_{\zeta ,\chi }^{\left( h\right)
}\left( m;\left( aj\right) _{b},b:q\right)
\end{equation*}%
\noindent for $p|b$ and $m+1\equiv 0\left( \text{mod}\left( p-1\right)
\right) $.
\end{proof}

\section{Analogues of Hardy-Berndt Type Sums}

\setcounter{equation}{0} \setcounter{theorem}{0}

As mentioned in Section 1, the classical Dedekind sums first arose in the
transformation formula of the logarithm of the Dedekind eta function. The
logarithms of the classical theta function are studied by Berndt \cite{B1}
and Goldberg \cite{G1} derived the transformation formulas for classical
theta-functions. Arising in the transformation formulas, there are six
different arithmetic sums, which are thus similar to Dedekind sums and
called as Hardy sums or Berndt's arithmetic sums. For $h$, $k\in \mathbb{Z}$
with $k>0$, these six sums are defined as follows:%
\begin{eqnarray*}
S\left( h,k\right) &=&\sum_{j=1}^{k-1}\left( -1\right) ^{j+1+\left[ \frac{hj%
}{k}\right] _{G}},\text{ }s_{1}\left( h,k\right) =\sum_{j=1}^{k}\left(
-1\right) ^{\left[ \frac{hj}{k}\right] _{G}}\left( \left( \frac{j}{k}\right)
\right) , \\
s_{2}\left( h,k\right) &=&\sum_{j=1}^{k}\left( -1\right) ^{j}\left( \left( 
\frac{j}{k}\right) \right) \left( \left( \frac{hj}{k}\right) \right) ,\text{ 
}s_{3}\left( h,k\right) =\sum_{j=1}^{k}\left( -1\right) ^{j}\left( \left( 
\frac{hj}{k}\right) \right) , \\
s_{4}\left( h,k\right) &=&\sum_{j=1}^{k-1}\left( -1\right) ^{\left[ \frac{hj%
}{k}\right] _{G}},\text{ }s_{5}\left( h,k\right) =\sum_{j=1}^{k}\left(
-1\right) ^{j+\left[ \frac{hj}{k}\right] _{G}}\left( \left( \frac{j}{k}%
\right) \right) .
\end{eqnarray*}

\noindent The analytic and arithmetical properties of these sums were given
by Berndt \cite{B1}, Berndt and Goldberg \cite{BG1}, Can \cite{C1}, Goldberg 
\cite{G1}, Meyer \cite{Meyer}, Simsek \cite{simjnt2003dede},
Sitaramachandrarao \cite{S1}.

In this section, we show the sums, defined by Definition \ref{def10}, yield
new type sums, which we call analogues of Hardy-Berndt type sums.

By taking $m=2$, $r=1$ and $\zeta =-1$ in (\ref{1,3}), we obtain%
\begin{equation}
H_{n}\left( x,-1\right) =\frac{2^{n+1}}{n+1}\left( B_{n+1}\left( \frac{x+1}{2%
}\right) -B_{n+1}\left( \frac{x}{2}\right) \right) .  \label{M1}
\end{equation}

\noindent From Definition \ref{def7}, it is clear that $\overline{H}%
_{n}\left( x+2,-1\right) =\overline{H}_{n}\left( x,-1\right) $. Since $%
\overline{B}_{n}\left( x\right) $ is periodic for any integer, (\ref{M1})
can be written in terms of these functions as%
\begin{equation*}
\overline{H}_{n}\left( x,-1\right) =\frac{2^{n+1}}{n+1}\left( \overline{B}%
_{n+1}\left( \frac{x+1}{2}\right) -\overline{B}_{n+1}\left( \frac{x}{2}%
\right) \right) .
\end{equation*}

\noindent From (\ref{A2}), it is easy to see that (\ref{Y2*}) is also valid
for the functions $\overline{B}_{n}\left( x\right) $, that is, we have%
\begin{equation}
m^{n-1}\sum_{j=0}^{m-1}\overline{B}_{n}\left( x+\frac{j}{m}\right) =%
\overline{B}_{n}\left( mx\right) .  \label{M4}
\end{equation}

\noindent From (\ref{M4}) for $m=2$, we get%
\begin{equation*}
\overline{B}_{n}\left( \frac{x+1}{2}\right) =2^{1-n}\overline{B}_{n}\left(
x\right) -\overline{B}_{n}\left( \frac{x}{2}\right) .
\end{equation*}

\noindent We therefore have%
\begin{equation}
\overline{H}_{n}\left( x,-1\right) =\frac{2}{n+1}\overline{B}_{n+1}\left(
x\right) -\frac{2^{n+2}}{n+1}\overline{B}_{n+1}\left( \frac{x}{2}\right) .
\label{M2}
\end{equation}

Now, taking $u=-1$ in Definition \ref{def10}, we obtain%
\begin{equation}
S_{n,-1}\left( h,k\right) =\sum_{a=0}^{k-1}\left( -1\right) ^{\frac{ha}{k}}%
\frac{a}{k}\overline{H}_{n}\left( \frac{ha}{k},-1\right) .  \label{M3}
\end{equation}

\noindent Substituting (\ref{M2}) into (\ref{M3}) with $x=\frac{ha}{k}$, we
have%
\begin{equation}
S_{n,-1}\left( h,k\right) =\sum_{a=0}^{k-1}\left( -1\right) ^{\frac{ha}{k}}%
\frac{a}{k}\left( \frac{2}{n+1}\overline{B}_{n+1}\left( \frac{ha}{k}\right) -%
\frac{2^{n+2}}{n+1}\overline{B}_{n+1}\left( \frac{ha}{2k}\right) \right) .
\label{M5}
\end{equation}

\noindent By using (\ref{M5}), we define the following new sums, which we
call analogues of Hardy-Berndt type sums.

\begin{definition}
\label{def20}For $n$, $h$, $k\in \mathbb{Z}$ with $\left( h,k\right) =1$ and 
$n\geqslant 0$, we define%
\begin{eqnarray*}
HB_{n,0}\left( h,k\right) &=&\sum_{a=0}^{k-1}\left( -1\right) ^{\frac{ha}{k}}%
\frac{a}{k}\overline{B}_{n+1}\left( \frac{ha}{k}\right) , \\
HB_{n,1}\left( h,k\right) &=&\sum_{a=0}^{k-1}\left( -1\right) ^{\frac{ha}{k}}%
\frac{a}{k}\overline{B}_{n+1}\left( \frac{ha}{2k}\right) .
\end{eqnarray*}
\end{definition}

\begin{remark}
\label{re2}Let $k$ be an odd integer. Then, we have the following relations:

(\textbf{i}) If $h$ is even, then%
\begin{equation*}
HB_{n,0}\left( h,k\right) =s_{n+1}\left( h,k\right) \text{ and }%
HB_{n,1}\left( h,k\right) =s_{n+1}\left( h,2k\right) ,
\end{equation*}

\noindent where $s_{n+1}\left( h,k\right) $ is given by (\ref{A1}).

(\textbf{ii}) Let $h$ be an odd integer. Then

(\textbf{a}) If $n+1$ is even, then%
\begin{equation*}
HB_{n,1}\left( h,k\right) =2^{-1-n}s_{2,n+1}\left( h,k\right) +\frac{n+1}{%
2^{n+2}}H_{n}\left( -1\right) .
\end{equation*}

(\textbf{b}) If $n+1$ is odd, then%
\begin{equation*}
HB_{n,1}\left( h,k\right) =\frac{1}{2}\left( 1-2^{-n}\right) s_{n+1}\left(
h,k\right) -\frac{1}{4}s_{5,n+1}\left( h,k\right) -2^{-n}s_{3,n+1}\left(
h,k\right) ,
\end{equation*}

\noindent where $s_{2,n+1}\left( h,k\right) $, $s_{3,n+1}\left( h,k\right) $
and $s_{5,n+1}\left( h,k\right) $ are generalizations of Berndt's arithmetic
sums $s_{2}\left( h,k\right) $, $s_{3}\left( h,k\right) $ and $s_{5}\left(
h,k\right) $, respectively (\cite{cancenkcikurthardy}).
\end{remark}

\textbf{Conclusion: }The conclusion we can most likely draw from above is
that the sums given by Definition \ref{def10} is different from Carlitz,
Apostol, Berndt type Dedekind sums, and Definition \ref{def19} is different
from Ota type Dedekind sums. For instance, Carlitz type Dedekind sums are
defined by Frobenius-Euler numbers $H_{n-1}\left( u^{-1}\right) $ as follows 
\cite{40}:%
\begin{equation*}
S\left( h,k:n\right) =\frac{n}{h^{n}}\sum_{u}\frac{H_{n-1}\left(
u^{-1}\right) }{\left( u-1\right) \left( u^{-h}-1\right) }.
\end{equation*}%
\noindent In Definition \ref{def10} and Definition \ref{def19}, we use
Frobenius-Euler functions, which provides a different and useful approach to
the theory of Dedekind sum.

Definition \ref{def12} is different from those of Kim, Rosen and Synder and
Kudo.

\textbf{Acknowledgment: }This work was supported by Akdeniz University
Scientific Research Projects Unit.

\end{document}